 \documentclass[final,12 pt, 3p]{elsarticle}

 \usepackage{graphics}
 \usepackage{verbatim }
 \usepackage{graphicx}
 \usepackage{epsfig}

\usepackage{amssymb}
\usepackage{amsmath}
\usepackage{amsthm}

   \newtheorem{theorem}{Theorem}[section]

   \newtheorem{example}[theorem]{Example}

\let\today\relax
\journal{Applied Mathematics and Computations}

\begin{document}

\begin{frontmatter}



\title{PDE-Based Optimization for Advection Diffusion Equation in 2D Domain}
\author{Yunfei Song}

\address{Quant Strategies Group, BA Securities\\
\large\emph{February 10, 2021}}

\begin{abstract}
In this paper,  we propose a PDE-based optimization motivated by the problem of microfluidic heat transfer to finding the optimal incompressible velocity fields in 2D domain. To solve this optimization model, we use spectral method to discretize the model to obtain an ODE based optimization.  This way significantly reduces the complexity of the discretization optimization model, and gives a more accurate approximation of the original PDE based optimization. Some theoretical results are obtained. 

\end{abstract}


\end{frontmatter}

\section{Introduction}
The \emph{advection-diffusion equation (ADE)} is given as follows:
\begin{equation}\label{ade}
\frac{\partial \phi}{\partial t}+\textbf{v}\cdot\nabla \phi= \kappa \nabla^2\phi,
\end{equation}
where $\textbf{v}$ is the velocity vector, $\kappa=\tfrac{1}{Pe}$ and $Pe$ is the Peclet number. 
The Peclet number is the ratio of the characteristic time for diffusion over the characteristic time for advectoin, and defined as $Pe=\frac{VL}{D}$ where $V$ is a characteristic lengthscale, $V$ is a characteristic velocity intensity, and $D$ is the diffusivity, respectively. The Peclet Number is a dimensionless parameter that indicates the relative importance of
advection and diffusion to the transport of scalars in a given system.  ADE is also referred to as convection-diffusion equation. ADE describes the motion of the combination of an advection process and a diffusion process. The second term on the left hand side in (\ref{ade}) is the advection (convection) term, the right hand side in (\ref{ade}) is the diffusion term.  If we add a reaction term $\phi$ on the right hand side in (\ref{ade}), then it is called advection-diffusion-reaction equation. In this paper, we focus on ADE in 2-dimensional (2-D) space.

A review paper on advection-diffusion equations in chaotic flows is \cite{cort1}. 
A spectral method is proposed to numerically solve advection-diffusion-reaction equations for reactive mixing problem with laminar chaotic flow in \cite{adrover1}. 
The exponent characterizing the decay towards the equilibrium distribution of a generic diffusing scalar advected by a nonlinear flow is studied in \cite{cerbelli1}.
The dynamics of a single irreversible reaction occuring in a bounded incompressible flow is analyzed in \cite{cerbelli2}.
Spectral methods used to solve partial differential equations are given in \cite{gottlieb}.The implementations of spectral methods to solve partial differential equations by using MATLAB are studied in \cite{trefethen}. The theoretical methods used in this paper use some results of partial differential equations found in \cite{evans}. 

PDE constrained optimization problem is a combination of differential equation and general optimization. The PDE part can be appeared in the objective function part, constraints part or both. There are usually two ways to solve a PDE based optimization, discretization-then-optimization or optimization-then-discretization. In the discretization-then-optimization method, one will discretize the differential equation by certain numerical discretization methods into a general large scale traditional optimization problem, then solving the optimization problem by general optimization algorithms, e.g., interior point method, e.g., \cite{ipopt},  to solve the  so called optimality conditions. Efficient algorithms for large-scale nonlinear optimization problems are refereed to as \cite{curtis1, curtis2, curtis3}.  In the optimization-then-discretization method, one will derive the so called optimality condition of the PDE based optimization, i.e., the sufficient and necessary condition such that the solution of the condition is the optimal solution of the PDE based optimization. Then one will use numerical discretization method to solve the optimality condition which involves some differential equations forms.  More details can be refereed to  \cite{biegler}, \cite{hinze}. Optimal transport of diffusive scalar from the boundary is studied in \cite{grover}, where uses PDE based optimization to solve heat transfer problem with the advection diffusion equation. 
Gibbs phenomena is a common oscillatory behavior when using spectral method for certain differential equations, e.g., \cite{songjung}. 

When the PDE is discretized by numerical methods, e.g., discretization methods or spectral methods, the PDE based optimization will become to an ODE based optimization, which is also named dynamical system based optimization. A key problem in this area is to find the invariant set that the dynamical system is never escaped, i.e., controlled by the set. Some condition for certain system is derived in \cite{song4, song1, songthesis, songnagumo}. Invariance preserving for using certain discretization methods are given in \cite{song3, song2, songthesis}.  By constructing certain invariant set by using Dikin ellipsoid is refereed to \cite{songdikin}.

In this paper, motivated by the problem of microfluidic heat transfer, we propose a PDE-based optimization to identify the optimal incompressible velocity fields in 2D domain. We applied the discretization-then-optimization method. In particular, we use spectral method to discretize the model to obtain an ODE based optimization to solve this optimization model.  The novelty of this paper is that this way significantly reduces the complexity of the discretization optimization model, and gives a more accurate approximation of the original PDE based optimization. Some theoretical results are obtained about the boundness of the model coefficients. 

\section{Basic Background}
Assume the domain of (\ref{ade}) is $\Omega\times[0,\infty), \Omega\in \mathbb{R}^2$. To make (\ref{ade}) to be solvable, we need to add the boundary and initial conditions. The boundary condition of (\ref{ade}) can be given either as Dirichlet (or first-type) boundary condition
\begin{equation}
\phi(x,t)=f(t) \text{ for all } x\in \partial \Omega,
\end{equation}
or Neumann (or second-type) boundary condition 
\begin{equation}\label{newman}
\frac{\partial \phi(x)}{\partial n}\Big|_{\partial \Omega}=\nabla \phi(x)\cdot \textbf{n}|_{\partial\Omega}=0.
\end{equation}

The initial condition is given as follows:
\begin{equation}
\phi(x,0)=g(x) \text{ for all } x\in \Omega.
\end{equation}

The velocity vector $\textbf{v}$ of the flow in 2-D  is generally obtained by  the following function $\psi(x,y)$ 
\begin{equation*}\label{stream}
  v_1(x,y)=-\frac{\partial \psi(x,y)}{\partial y}, ~~v_2(x,y)=\frac{\partial \psi(x,y)}{\partial x}.
\end{equation*}
The function $\psi(x,y)$ is called stream function, e.g., \cite{cort1}. 
Under certain conditions, the stream function $\psi(x,y)$ satisfies the biharmonic equation, i.e., $\nabla^4\phi=0$.

We now present  a well-known properties of ADE.
\begin{theorem}\label{thm1}
Assume ADE is given as (\ref{ade}) and the Neumann boundary condition is given as (\ref{newman}), then 
\begin{equation}
  \frac{d\|\phi\|_{L^2}^2}{dt}=-2\kappa \|\nabla\phi\|_{L^2}^2
\end{equation}
\end{theorem}

\begin{proof}
By multiplying $\phi$ by both sides of (\ref{ade}) and take the integration over $\Omega$, we have
\begin{equation}\label{xmys}
  \int_{\Omega}\phi\frac{\partial \phi}{\partial t}d\textbf{x}+\int_{\Omega}\phi \textbf{v}\cdot \nabla \phi d\textbf{x}=\kappa \int_{\Omega}\phi\nabla^2\phi d\textbf{x}.
\end{equation}
For the first term in the left side of (\ref{xmys}), we have 
\begin{equation}\label{eqms}
  \int_{\Omega}\phi\frac{\partial \phi}{\partial t}d\textbf{x}=\frac{1}{2}\int_{\Omega}\frac{\partial \phi^2}{\partial t}d\textbf{x}=\frac{1}{2}\frac{d}{dt}\int_{\Omega}\phi^2d\textbf{x}=\frac{1}{2}\frac{d\|\phi\|^2_{L^2}}{dt}.
\end{equation}
For the second term in the left side of  (\ref{xmys}), by using the divergence theorem, e.g., \cite{evans}, and noting that the flow $\textbf{v}$ is incompressible, we have
\begin{equation}\label{eqms22}
\begin{split}
  \int_{\Omega}\phi \textbf{v}\cdot \nabla \phi d\textbf{x}&=\int_{\Omega}\textbf{v}\cdot\phi  \nabla \phi d\textbf{x}=\frac{1}{2}\int_{\Omega}\textbf{v}\cdot\nabla \phi^2 d\textbf{x}=
\frac{1}{2}\int_{\Omega}\nabla\cdot\textbf{v} \phi^2 d\textbf{x}\\
&=\frac{1}{2}\int_{\partial\Omega}\phi^2\textbf{v}\cdot\textbf{n}d\sigma=0.
\end{split}
\end{equation}
For the right side of (\ref{xmys}), we have 
\begin{equation}\label{eqms33}
\begin{split}
  \int_{\Omega}\phi\nabla^2\phi d\textbf{x}&=\int_{\Omega}\nabla\cdot(\phi\nabla\phi)d\textbf{x}-
\int_{\Omega}\nabla\phi\cdot\nabla\phi d\textbf{x}\\
&=\int_{\partial \Omega}\phi\nabla\phi\cdot\textbf{n}d\textbf{x}-
\int_{\Omega}\nabla\phi\cdot\nabla\phi d\textbf{x}=-\int_{\Omega}\nabla\phi\cdot\nabla \phi d\textbf{x}=-\|\nabla\phi\|_{L^2}^2.
\end{split}
\end{equation}
Substituting (\ref{eqms}), (\ref{eqms22}) and (\ref{eqms33}) into (\ref{xmys}), the theorem is immediate. 
\end{proof}
Theorem \ref{thm1} indicates that the magnitude of flux $\phi$ decays as time goes forward when $\kappa$ is positive.  Also, if $\|\nabla\phi\|_{L^2}$ has a lower bound, then $\phi$ approaches to 0 as time goes to infinity. 

Spectral methods are powerful methods used for numerically finding the solution of ordinary or partial differential equations. Unlike the scheme based numerical methods, e.g., finite difference methods, spectral methods are global methods, where the computation at any given point depends not only on information at neighboring points, but on information from the entire domain.
The key idea of \emph{Spectral Galerkin methods} is the fact that a smooth function $\phi(x,t)$ can be represented in the form of
\begin{equation*}
\phi(x,t)=\sum_{n=0}^\infty\hat a_n(t)\psi_n(x),
\end{equation*}
where $\{\psi_n(x)\}$ are the global basis functions. We truncate $\phi(x,t)$ to the first $N+1$ terms as 
\begin{equation}\label{eqd1}
\phi_N(x,t)=\sum_{n=0}^N\hat a_n(t)\psi_n(x),
\end{equation}
and we have $\phi_N(x,t)\rightarrow \phi(x,t)$ as $N\rightarrow\infty,$ i.e., finding the approximation so that the residual is orthogonal to the space from which $\phi_N(x,t)$ comes. 
By substituting (\ref{eqd1}) into (\ref{ade}), we have 
\begin{equation*}\label{dcd2d}
R_N(x,t) = \sum_{n=0}^N\frac{\partial \hat a_n(t)}{\partial t}\psi_n(x)+\sum_{i=0}^N\hat a_n(t)\textbf{v}\cdot\nabla\psi_n(x)-\kappa\sum_{i=0}^N\hat a_n(t)\nabla^2\psi_n(x).
\end{equation*}
Applying the test function $\psi_i(x)$ to (\ref{dcd2d}) under the weight function $w(x)$ and taking the integration, we have 
\begin{equation*}\label{dcc2d}
\int_{\Omega}R_N(x,t) \psi_i(x) w(x)dx \text{  for all } i \in \{1,2,...,N\}.
\end{equation*}

\section{Build PDE based Optimization Model}
Note that in the following discussion, if the exact norm is not explicitly given, then  we are using $L_2$ norm in our final optimization problem. This is since $L_2$ norm  will yield quadratic form when the discretization method is used to transform the PDE based optimization problem into standard optimization problem, which is usually easy to solve by current optimization solvers. 

First, we consider the objective function used in our problem. The following cost functions can be considered to achieve the purpose:

\begin{equation}
  \min \|\phi(x,y,t)\|_{L_1}
\end{equation}

\begin{equation}
  \min \|\phi(x,y,t)\|_{L_2}
\end{equation}

\begin{equation}
  \min  t, \text{ s.t. } \|\phi(x,y,t)\|_{L_1}=0.99\|\phi(x,y,0)\|
\end{equation}
But in our research, we use the following objective function:
\begin{equation}
  \min_t Q(t)=\|\nabla\phi(x,y,t)\|_{L_2}^2
\end{equation}
where 
\begin{equation}
\|\nabla\phi\|_{L_2}^2=\int_0^1\int_0^1\nabla\phi\cdot\nabla\phi dxdy.
\end{equation}
 This metric is directly the measure of variance in the domain. Minimizing this at a certain time will require the history of all controls up to that time.
An alternative measure as cost function is given as follows:

\begin{equation}
  \max P(t)=\frac{dQ(t)}{dt}
\end{equation}
To maximize this implies we are choosing a velocity field which maximizes the instantaneous decay of the variance of the scalar. Theoretically, this optimal velocity field does not depend on the history.

We now consider the involved constraints in our problem. The first constraint is obviously the advection-diffusion equation:
\begin{equation}\label{adeopt}
\frac{\partial \phi}{\partial t}+\textbf{v}\cdot\nabla \phi= \kappa \nabla^2\phi.  
\end{equation}
The boundary condition of the ADE (\ref{adeopt}) is Dirichlet type rather than Newmann type, since we need the boundary of the region has fixed tempereture as input flux. 
The velociy flow in (\ref{adeopt}) is now variables rather than given functions. From the physical perspective, we also assume that the velocity flow is incompressible:
\begin{equation}
  \nabla\cdot\textbf{v}=\frac{\partial v_1}{\partial x}+\frac{\partial v_2}{\partial y}=0.
\end{equation}
The velocity flow is required to normalize in terms of either one of the following two manners:
\begin{equation}\label{vflowcase}
  \|\textbf{v}\|_{L_2}^2=1~~
\text{ or }~~
  \|\nabla \textbf{v}\|_{L_2}^2=1.
\end{equation}
We also need the velocity flow vanishes, i.e., no-slip, on the boundary of the region.
\begin{equation}\label{noslip}
  \textbf{v}|_\Omega=0.
\end{equation}

According to the discussion above, the overal PDE based optimization problem can be summarized as follows:
\begin{equation}\label{pdeopt}
  \begin{array}{rrl}
    \min &\|\nabla\phi\|_{L_2}^2&\\
\text{ s.t.} & \frac{\partial \phi}{\partial t}+\textbf{v}\cdot\nabla \phi&= \kappa \nabla^2\phi \\ 
& \frac{\partial v_1}{\partial x}+\frac{\partial v_2}{\partial y}&=0\\
& \|v\|_{L_2}^2&=1\\
& \textbf{v}(x,y)&=0~~~\text{ for } (x,y)\in \partial [0,1]\times \partial [0,1]\\
& \phi(x,y,t)&=C~~\text{ for } (x,y)\in \partial [0,1]\times \partial [0,1]\\
  \end{array}
\end{equation}

\section{Spectral Method to PDE}
For the computational reasons, the velocity flow $v$ is usually defined as in terms of the finite truncated basis functions 

\begin{equation}\label{expand3d}
  v_1(x,y)=\sum_{k,\ell=0}^M\alpha_{k\ell}(t)\hat v_{k\ell}(x,y)~\text{ and }~ v_2(x,y)=\sum_{k,\ell=0}^M\beta_{k\ell}(t)\hat u_{k\ell}(x,y),
\end{equation}
where $\hat v_{k\ell}(x,y)$ and $\hat u_{k\ell}(x,y)$ are some basis functions, e.g., 
\begin{equation}\label{bafun1}
  \hat v_{k\ell}(x,y)=\sin(k\pi x)\cos(\ell\pi y) \text{ and } \hat u_{k\ell}(x,y)=\cos(k\pi x)\sin(\ell\pi y).
\end{equation}
Then we will see that the coefficients, $\alpha_{k\ell}, \beta_{k\ell},$ shown in (\ref{expand3d}) are the introduced variables in our discretization standard optimization problem. If the basis function is given as in (\ref{bafun1}), then, according to the stream function relationship (\ref{stream}), we need to add the following conditions as constraints to our optimization problem:
\begin{equation}\label{addequ2}
  k\alpha_{k\ell}(t)+\ell\beta_{k\ell}(t)=0, \text{ for all } k,\ell=0,1,...,M.
\end{equation}

 We also note that if the basis functions of the spectral methods used to solve the PDE are well chosen, then the boundary condition of the PDE is automatically satisfied, which yields that it can be eliminated in our optimization problems.  
Since we are using Dirichlet boundary condition and the region of the PDE is $[0,1]\times [0,1]$, we choose $\{\sin(i\pi x)\sin(j\pi y)\}_{i,j=0}^\infty$ as the basis functions used in the spectral method. Since the sine function equals to 0 for $i,j=0$, we obtain that  $\phi(x,y,t)$ can be represented as follows:
\begin{equation}\label{pdespec}
  \phi(x,y,t)=\sum_{i,j=1}^Na_{ij}(t)\sin(i\pi x)\sin(j\pi y),
\end{equation}
where $N$ denotes the number of modes of the spectral method. 
Applying (\ref{pdespec}) to the gradient of $\phi$, we have
\begin{equation}
  \nabla \phi(x,y,t)=\left(\sum_{i,j=1}^N a_{ij}(t)i\pi\cos(i\pi x)\sin(j\pi y),\sum_{i,j=1}^N a_{ij}(t)j\pi\sin(i\pi x)\cos(j\pi y)\right)^T.
\end{equation}
Then the objective function in the optimization problem (\ref{pdeopt}) yields 
\begin{equation}\label{objsimple}
\begin{array}{rl}
  \|\nabla\phi\|_{L^2}^2&=\int_{0}^1\int_0^1\nabla\phi\cdot\nabla \phi dxdy\\
&=\pi^2\sum\limits_{i,j=1}^Na_{ij}^2(t)\int_{0}^1\int_0^1(i^2\cos^2(i\pi x)\sin^2(j\pi y)+j^2\sin^2(i\pi x)\cos^2(j\pi y))dxdy\\
&=\frac{\pi^2}{4}\sum\limits_{i,j=1}^Na_{ij}^2(t)(i^2+j^2).
\end{array}
\end{equation}

We now consider to use spectral method to solve the ADE constraint in (\ref{pdeopt}). Subsituting  (\ref{expand3d}) and (\ref{pdespec}) into the ADE constraint (for space consideration, we move the diffusion term to the left hand side, and move the advection term to the right hand side), we have 
\begin{equation}\label{adedisc}
  \begin{split}
    &\sum_{ij}\frac{da_{ij}(t)}{dt}\sin(i\pi x)\sin(j\pi y)+\kappa\pi^2\sum_{ij}a_{ij}(t)(i^2+j^2)\sin(i\pi x)\sin(j\pi y)\\
=&-\sum_{k\ell}\alpha_{k\ell}(t)\sin(k\pi x)\cos(\ell \pi y)\sum_{ij}a_{ij}(t)(i\pi)\cos(i\pi x)\sin(j\pi y)\\
&-\sum_{k\ell}\beta_{k\ell}(t)\cos(k\pi x)\sin(\ell \pi y)\sum_{ij}a_{ij}(t)(j\pi)\sin(i\pi x)\cos(j\pi y).
  \end{split}
\end{equation}
Multiplying the test function $\sin(m\pi x)\sin(n\pi y)$ into the both hand sides of (\ref{adedisc}) and then integrating from 0 to 1, we have 
\begin{equation}\label{odesysm}
  \frac{da_{mn}(t)}{dt}+\kappa\pi^2a_{mn}(t)(m^2+n^2)=-4\pi\sum_{k\ell ij}(iA_{k\ell ij}^{mn}\alpha_{k\ell}(t)+jB_{k\ell ij}^{mn}\beta_{k\ell})a_{ij}(t),
\end{equation}
where 
\begin{equation}\label{avalue}
  A_{k\ell ij}^{mn}=\int_0^1\sin(m\pi x)\sin(k\pi x)\cos(i\pi x)dx\int_0^1\sin(n\pi y)\cos(\ell \pi y)\sin(j\pi y)dy,
\end{equation}
\begin{equation}\label{bvalue}
  B_{k\ell ij}^{mn}=\int_0^1\sin(m\pi x)\cos(k\pi x)\sin(i\pi x)dx\int_0^1\sin(n\pi y)\sin(\ell \pi y)\cos(j\pi y)dy.
\end{equation}
According to (\ref{addequ2}), we can simplify (\ref{odesysm}) as
\begin{equation}\label{adesimplespc}
  \frac{da_{mn}(t)}{dt}+\kappa\pi^2a_{mn}(t)(m^2+n^2)=-4\pi\sum_{ij}\big(\sum_{k\ell}(iA_{k\ell ij}^{mn}-j\frac{k}{\ell}B_{k\ell ij}^{mn})\alpha_{k\ell}(t)\big)a_{ij}(t)
\end{equation}
The values of $A_{k\ell ij}^{mn}$ and $B_{k\ell ij}^{mn}$ are easy to compute before solving the ODE system (\ref{odesysm}). For the sake of simplicity, we denote (\ref{odesysm}) as follows:
\begin{equation}\label{simpleode}
  \frac{da_{mn}(t)}{dt}-Q_{mn}\textbf{a}=0 \text{ for all } m, n= 1,2,...,N,
\end{equation}
where $Q_{mn}$ is the matrix which is easily obtained according to (\ref{adedisc}).

For the second constraint in the optimization problem (\ref{pdeopt}), we can show that it is redundant by the relationship between the velocity flow and the stream function. In fact, according to (\ref{stream}), we have 
\begin{equation}
\frac{\partial v_1}{\partial x}+\frac{\partial v_2}{\partial y}=-\frac{\partial^2\psi(x,y)}{\partial x\partial y}+\frac{\partial^2\psi(x,y)}{\partial x\partial y}=0.
\end{equation}

We then analyze the third constraint in the optimization problem (\ref{pdeopt}). We have
\begin{equation}\label{normsin}
\begin{split}
\|\textbf{v}\|_{L^2}^2&=\int_0^1\int_0^1(v_1^2+v_2^2)dxdy\\
&=\sum\limits_{k,\ell=1}^M\alpha_{k\ell}^2\int_0^1\int_0^1\sin^2(k\pi x)\cos^2(\ell\pi y)+\beta_{k\ell}^2\cos^2(k\pi x)\sin^2(\ell\pi y)dxdy\\
&=\frac{1}{4}\sum\limits_{k,\ell=1}^M\alpha_{k\ell}^2+\beta_{k\ell}^2=1.
\end{split}
\end{equation}
Thus, according to (\ref{addequ2}), the third constraint in (\ref{pdeopt}) yields
\begin{equation}\label{simplenorm}
\sum_{k,\ell=1}^M\left(1+\frac{k^2}{\ell^2}\right)\alpha_{k\ell}^2(t)=4.
\end{equation}

We now analyze the case if the third constriant in the optimization problem is the second case of (\ref{vflowcase}). Recall that 
\begin{equation}\label{normsin2}
\begin{split}
  \|\nabla \textbf{v}\|_{L^2}^2&=\int_0^1\int_0^1\left(\frac{\partial v_1}{\partial x}\right)^2+\left(\frac{\partial v_1}{\partial y}\right)^2+\left(\frac{\partial v_2}{\partial x}\right)^2+\left(\frac{\partial v_2}{\partial y}\right)^2dxdy\\
&=\frac{\pi^2}{4}\sum_{k,\ell=1}^M(k^2+\ell^2)(\alpha_{k\ell}^2+\beta_{k\ell}^2)=1.
\end{split}
\end{equation}

According to (\ref{objsimple}), (\ref{adesimplespc}),and (\ref{simplenorm}), the current optimization model by using spectral method is given as follows:
\begin{equation}\label{optdse}
  \begin{array}{rl}
    \min & \frac{\pi^2}{4}\sum_{i,j=1}^Na_{ij}^2(t_{final})(i^2+j^2)\\
\text{s.t.}&\frac{da_{mn}(t)}{dt}+\kappa\pi^2a_{mn}(t)(m^2+n^2)\\
&~~~~~~~~~=-4\pi\sum_{ij}\big(\sum_{k\ell}(iA_{k\ell ij}^{mn}-j\frac{k}{\ell}B_{k\ell ij}^{mn})\alpha_{k\ell}(t)\big)a_{ij}(t)\text{ for all } m,n\\
&\sum_{k,\ell=1}^M\big(1+\frac{k^2}{\ell^2}\big)\alpha_{k\ell}^2(t)=4\\
  \end{array}
\end{equation}
We now rewrite (\ref{optdse}) in the following bilinear form
\begin{equation}\label{optmatrix}
  \begin{array}{rl}
    \min & \textbf{a}(t_{final})^TQ\textbf{a}(t_{final})\\
    \text{s.t.} &\frac{d\textbf{a}(t)}{dt}=-D\textbf{a}(t)-(I\otimes\alpha(t)^T)R(e\otimes\textbf{a}(t))\\
                &\alpha(t)^TZ\alpha(t)=4
  \end{array}
\end{equation}
where 
\begin{equation}\label{parmt}
\begin{array}{rll}
\textbf{a}(t)&=(a_{11}(t),...,a_{1N}(t),...,a_{N1}(t),...,a_{NN}(t))^T&\in \mathbb{R}^{N^2\times 1}\\
\alpha(t)&=(\alpha_{11}(t),...,\alpha_{1M}(t),...,\alpha_{M1}(t),...,\alpha_{MM}(t))^T&\in \mathbb{R}^{M^2\times 1}\\
 Q&=\frac{\pi}{4}\text{diag}\{i^2+j^2\}_{i=1,...,N,j=1,...,N} & Q\in \mathbb{R}^{N^2\times N^2}\\
D&= \kappa \pi^2 \text{diag}\{m^2+n^2\}_{m=1,...,N,n=1,...,N} &D\in \mathbb{R}^{N^2\times N^2}\\ 
I&=\text{diag}\{1,1,...,1\} & I \in \mathbb{R}^{N^2\times N^2}\\
e&=[1,1,...,1]^T & e\in \mathbb{R}^{N^2\times1}\\
R&=\text{diag}\{C^{mn}\}_{m=1,...,N,n=1,...,N}& R\in \mathbb{R}^{M^2N^2\times N^4}\\
C^{mn}&=[4\pi(iA_{k\ell ij}^{mn}-j\frac{k}{\ell}B_{k\ell ij}^{mn})]_{k,\ell=1,2,...,M, i,j=1,2,...,N} &C^{mn}\in\mathbb{R}^{M^2\times N^2} \\
Z&= \text{diag}\{1+\frac{k^2}{\ell^2}\}_{k=1,...,M,\ell=1,...,M} &Z\in \mathbb{R}^{M^2\times M^2}
\end{array}
\end{equation}
The  formulate corresponding to the advection part in (\ref{parmt}) is  complicated. We can consider it in this way. First, note that the indices $(i,j),(k,\ell),$ and $(m,n)$ are due to the ADE is defined on 2 diamensional space, we can combine them. Then the first constraint in (\ref{optdse}) yields (here we ignore the diffusion part for simplicity.)
\begin{equation}\label{simpode2}
  \frac{da_m(t)}{dt}=-4\pi\sum_{i}\sum_kC_{ki}^m\alpha_k(t)a_i(t)=-4\pi \alpha(t)C^m\textbf{a}(t), m=1,2,...,N^2,
\end{equation}
where $\alpha(t),\textbf{a}(t),C_{ki}^m,$ and $C^m$ are given as the same manner as (\ref{parmt}). Writing all (\ref{simpode2}) in a matrix form, we have 
\begin{equation}
  \frac{d\textbf{a}(t)}{dt}=-4\pi (I\otimes \alpha(t))\text{diag}\{C^m\}(e\otimes \textbf{a}(t)),
\end{equation}
where $I$ and $e$ are in appropriate dimension.

\section{Theoretical Analysis}
We now give some theoretical properties of the constriants. First, let us consider the (\ref{odesysm}) and (\ref{normsin}). Denote
\begin{equation}\label{fgnote}
  F_{ij}^{mn}(t)=\sum_{k,\ell}iA_{k\ell ij}^{mn}\alpha_{k\ell}(t), ~~G_{ij}^{mn}(t)=\sum_{k,\ell}jB_{k\ell ij}^{mn}\beta_{k\ell}(t),
\end{equation}
then, according to Cauchy's inequality,  we have 
\begin{equation}
    (F_{ij}^{mn}(t))^2\leq i^2\big(\sum_{k,\ell}(A_{k\ell ij}^{mn})^2\big)\big(\sum_{k,\ell}\alpha_{k\ell}^2\big),~~
    (G_{ij}^{mn}(t))^2\leq j^2\big(\sum_{k,\ell}(B_{k\ell ij}^{mn})^2\big)\big(\sum_{k,\ell}\beta_{k\ell}^2\big),
\end{equation}
which, according to (\ref{normsin}), yields
\begin{equation}\label{coeffer}
\begin{split}
  (F_{ij}^{mn}(t)+G_{ij}^{mn}(t))^2&\leq 2(F_{ij}^{mn}(t))^2+2(G_{ij}^{mn}(t))^2\\
&\leq 2\max\big\{i^2\sum_{k,\ell}(A_{k\ell ij}^{mn})^2, j^2\sum_{k,\ell}(B_{k\ell ij}^{mn})^2\big\}\big(\sum_{k,\ell}\alpha_{k\ell}^2+\sum_{k,\ell}\beta_{k\ell}^2\big)\\
&\leq 8\max\big\{i^2\sum_{k,\ell}(A_{k\ell ij}^{mn})^2, j^2\sum_{k,\ell}(B_{k\ell ij}^{mn})^2\big\}.
\end{split}
\end{equation}
Denote 
\begin{equation}
K_{ij}^{mn}=\max\big\{i\big(\sum_{k,\ell}(A_{k\ell ij}^{mn})^2\big)^{1/2}, j\big(\sum_{k,\ell}(B_{k\ell ij}^{mn})^2\big)^{1/2}\big\},
\end{equation} 
then, according to (\ref{coeffer}), we have 
\begin{equation}
  -2\sqrt{2}K_{ij}^{mn}\leq F_{ij}^{mn}(t)+G_{ij}^{mn}(t)\leq 2\sqrt{2}K_{ij}^{mn}.
\end{equation}
Note that $F_{ij}^{mn}(t)+G_{ij}^{mn}(t)$ is the coefficient of $a_{ij}(t)$ in (\ref{odesysm}). Thus if we write (\ref{odesysm}) with all $m,n$ together in the form of 
\begin{equation}\label{odeform}
\frac{d\textbf{a}}{dt}=-A(t)\textbf{a} -D\textbf{a},  
\end{equation}
where $A(t)$ is the matrix corresponding to advection part, and $D$ is the matrix corresponding to diffusion part\footnote{The advection part $A(t)$ depends on $t$ since it involves $\alpha(t)$ and $\beta(t)$, while the diffusion part $D$ is constant.}. Denote $K=\max_{ijmn}\{K_{ij}^{mn}\},$ then we have that 
\begin{equation}
  -8\sqrt{2}\pi K\leq A_{ij} \leq  8\sqrt{2}\pi K,
\end{equation}
which means the matrix $A(t)$ corresponding to advectin part is a bounded matrix which is independent on $t$. According to (\ref{avalue}) and (\ref{bvalue}), we can compute that most of $A_{k\ell ij}$ and $B_{k\ell ij}$ are equal to zero, thus, it is easy to show that there exists a constant $\gamma>0$, such that
\begin{equation}\label{equde}
  0<K\leq \gamma N,
\end{equation}
where $N$ is the number of modes used in the spectral method. \\

We now consider (\ref{odesysm}) and (\ref{normsin2}), and we still use the notation in (\ref{fgnote}). According to Cauchy's inequality, we have 
\begin{equation}
\begin{split}
    (F_{ij}^{mn}(t))^2&\leq i^2\big(\sum_{k,\ell}(k^2+\ell^2)^{-1}(A_{k\ell ij}^{mn})^2\big)\big(\sum_{k,\ell}(k^2+\ell^2)\alpha_{k\ell}^2\big),\\
    (G_{ij}^{mn}(t))^2&\leq j^2\big(\sum_{k,\ell}(k^2+\ell^2)^{-1}(B_{k\ell ij}^{mn})^2\big)\big(\sum_{k,\ell}(k^2+\ell^2)\beta_{k\ell}^2\big),
\end{split}
\end{equation}
which, according to (\ref{normsin2}), yields
\begin{equation}\label{coeffe2r}
\begin{split}
  (F_{ij}^{mn}(t)+G_{ij}^{mn}(t))^2&\leq 2(F_{ij}^{mn}(t))^2+2(G_{ij}^{mn}(t))^2\\
&\leq \frac{8}{\pi^2}\max\big\{i^2\sum_{k,\ell}(k^2+\ell^2)^{-1}(A_{k\ell ij}^{mn})^2, j^2\sum_{k,\ell}(k^2+\ell^2)^{-1}(B_{k\ell ij}^{mn})^2\big\}.
\end{split}
\end{equation}
Denote 
\begin{equation}\label{hatKg}
\hat K_{ij}^{mn}=\max\big\{i\big(\sum_{k,\ell}(k^2+\ell^2)^{-1}(A_{k\ell ij}^{mn})^2\big)^{1/2}, j\big(\sum_{k,\ell}(k^2+\ell^2)^{-1}(B_{k\ell ij}^{mn})^2\big)^{1/2}\big\},
\end{equation} 
then, according to (\ref{coeffe2r}), we have 
\begin{equation}
  -2\sqrt{2}\hat K_{ij}^{mn}\leq F_{ij}^{mn}(t)+G_{ij}^{mn}(t)\leq 2\sqrt{2}\hat K_{ij}^{mn}.
\end{equation}
If the ODE system is also given as in (\ref{odeform}), and denote $\hat K=\max_{ijmn}\{\hat K_{ij}^{mn}\},$ then we have 
\begin{equation}
  -8\sqrt{2}\pi \hat K\leq A_{ij} \leq  8\sqrt{2}\pi \hat K.
\end{equation}
It is easy to show that there exists a constant $\hat \gamma>0$, such that 
\begin{equation}\label{equde1}
  0<\hat K\leq\hat\gamma.
\end{equation}
We note that there is no $N$ in (\ref{equde1}) which is different from the case in (\ref{equde}). This is  $\hat K_{ij}^{mn}$ defined as in (\ref{hatKg}) has $k^2+\ell^2$ in each term.

\begin{example}
Let the velocity vector in (\ref{ade}) be given as $\textbf{v}=(\sin(2\pi y),0)^T$, and the boundary and initial conditions be given as follows:
\begin{equation}
\begin{array}{rl}
\frac{\partial \phi}{\partial t}+\sin(2\pi y)\frac{\partial \phi}{\partial x}&=\kappa (\frac{\partial^2 \phi}{\partial x^2}+\frac{\partial^2 \phi}{\partial y^2})\\
\phi(x,y,t)&=0 ~~~~~~~\text{ for } (x,y)\in \partial [0,1]\times\partial[0,1]\\
\phi(x,y,0)&=f(x,y) ~~~\text{ for } (x,y)\in [0,1]\times[0,1].\\
\end{array}
\end{equation}

\end{example}

Let us consider the basis functions $\{\psi_{m,n}(x,y)=\sin(m\pi x)\sin(n\pi y)\}_{m,n=-\infty}^\infty.$
Then we have the following properties:
\begin{equation}
\int_{0}^1\sin^2(m\pi x)dx=
\begin{cases} 
1 &\mbox{if } m = 0 \\ 
 \tfrac{1}{2} & \mbox{if } m \geq1. 
 \end{cases}
\end{equation}
\begin{equation}
\psi_{m,n}(x,y)=0~~~\text{ for } (x,y) \in \partial [0,1] \times \partial [0,1].
\end{equation}

Let $\phi(x,y,t)$ have the following spectral represenation
\begin{equation}\label{edc3}
\phi(x,y,t)=\sum_{m=-\infty}^\infty\sum_{n=-\infty}^\infty\hat a_{m,n}(t)\psi_{m,n}(x,y),
\end{equation}
then
\begin{equation}
\hat a_{m,n}(t)=\frac{1}{\sigma_m\sigma_n}\int_{0}^1\phi(x,y,t)\psi_{m,n}(x,y)dxdy,
\end{equation}
where $\sigma_m=1$ for $m=0$ and $\sigma_m=\tfrac{1}{2}$ for $m\geq1.$

Subsituting (\ref{edc3}) into ADE we have 
\begin{equation}\label{cdes3}
\begin{split}
\sum_{m,n=0}^\infty\frac{d\hat a_{m,n}}{dt}\psi_{m,n}+\sin(2\pi y)\sum_{m,n=0}^\infty
\hat a_{m,n}\frac{\partial \psi_{m,n}}{\partial x}&=
\kappa\sum_{m,n=0}^\infty\hat a_{m,n}(\frac{\partial^2\psi_{m,n}}{\partial x^2}+\frac{\partial^2\psi_{m,n}}{\partial y^2})\\
\end{split}
\end{equation} 
Note that 
\begin{equation}
  \frac{\partial \psi_{m,n}}{\partial x}=m\pi\cos(m\pi x)\sin(n\pi y)
\end{equation}
\begin{equation}
  \frac{\partial^2 \psi_{m,n}}{\partial x^2}=-m^2\pi^2\psi_{m,n},~~~\frac{\partial^2 \psi_{m,n}}{\partial y^2}=-n^2\pi^2\psi_{m,n}.
\end{equation}
\begin{equation}
  \int_{0}^1 \psi_{k,\ell}\psi_{m,n}dxdy=\delta_{k,m}\delta_{\ell,n}\sigma_k\sigma_\ell.
\end{equation}

By multilying $\psi_{k,\ell}$ into both sides of (\ref{cdes3}), and taking the integral from 0 to 1, we have
\begin{equation}
\begin{split}
 & \frac{d\hat a_{k,\ell}}{dt}+\sum_{m,n=0}^\infty
\hat a_{m,n}\frac{m\pi}{\sigma_k\sigma_\ell}\int_{0}^1\sin(k\pi x)\cos(m\pi x)dx\int_{0}^1\sin(\ell\pi y)\sin(2\pi y)\sin(n\pi y)dy\\
=~~&
-\kappa (k^2+\ell^2)\pi^2\hat a_{k,\ell}
\end{split}
\end{equation}
By using Mathematica, we have 
\begin{equation}
  \int_{0}^1\sin(k\pi x)\cos(m\pi x)dx=
\begin{cases} 
\tfrac{k-k\cos(m\pi)\cos(k\pi)}{\pi(k^2-m^2)} &\mbox{if } k\neq m \\ 
 0 & \mbox{if } k=m. 
 \end{cases}
\end{equation}
\begin{equation}
  \int_{0}^1\sin(\ell\pi y)\sin(2\pi y)\sin(n\pi y)dy=
\begin{cases} 
\tfrac{4\ell n(-1+\cos(\ell\pi)\cos(n\pi))}{(-2+\ell-n)(2+\ell-n)(-2+\ell+n)(2+\ell+n)\pi} &\mbox{if } ok \\ 
 0 & \mbox{if } not ok. 
 \end{cases}
\end{equation}
The initial condition of the ODE is given as follows:
\begin{equation}
  \hat a_{k,\ell}(0)=\frac{1}{\sigma_k\sigma_\ell}\int_{0}^1f(x,y)\psi_{k,\ell}dx dy
\end{equation}

Thus, the ODE system is summarized as follows:
\begin{equation}
  \begin{split}
     \frac{d\hat{\textbf{a}}}{dt}&=(-A+D)\textbf{a}\\
       \hat{\textbf{a}}(0)&=\frac{1}{\sigma}\int_{0}^1f(x,y)\psi dxdy.\\
  \end{split}
\end{equation}

In that paper \cite{liuwj1}, the example is given as follows:
The domain is $[0,1]\times[0,1]$ and $\kappa=0.001$. The velocity flow is given as follows:
\begin{equation}
  v_1(x,y,t)=
  \begin{cases}
    \sin(\pi x)\cos(\pi y) & \text{ if } n\leq t< n+0.75 \\
-\sin(2\pi x)\cos(\pi y)& \text{ if } n+0.75\leq t< n+1
  \end{cases}
 \end{equation}
\begin{equation}
  v_2(x,y,t)=
  \begin{cases}
    -\cos(\pi x)\sin(\pi y) & \text{ if } n\leq t < n+0.75\\
2\cos(2\pi x)\sin(\pi y) & \text{ if } n+0.75\leq t< n+1
  \end{cases}
\end{equation}
The initial condition of the PDE is given as follows:
\begin{equation}
  \phi(x,y,0)=
  \begin{cases}
    1&\text{ if } 0\leq x\leq \tfrac{1}{2} \text{ and } 0\leq y\leq 1\\
   0 &\text{ if } \tfrac{1}{2}<x\leq 1 \text{ and } 0\leq y\leq 1
  \end{cases}
\end{equation}
The error is cacluated as follows:
\begin{equation}
  V(t)=\|\phi(x,y,t)-<\phi(x,y,0)>\|^2,
\end{equation}
where $<\phi>$ is the average of $\phi$. 

Again, we use the following spectral decomposition:
\begin{equation}
  \phi(x,y,t)=\sum_{i=1}^\infty\sum_{j=1}^\infty a_{ij}(t)\cos(i\pi x)\cos(j\pi y).
\end{equation}
First, let us calculate $<\phi(x,y,0)>$. We have
\begin{equation}
  \begin{split}
    <\phi(x,y,0)>&=\int_{0}^1\int_{0}^1\sum_{i=1}^\infty\sum_{j=1}^\infty a_{ij}(0)\cos(i\pi x)\cos(j\pi y)dxdy\\
&=\sum_{i=1}^\infty\sum_{j=1}^\infty a_{ij}(0)\int_{0}^1\int_{0}^1\cos(i\pi x)\cos(j\pi y)dxdy\\
&=a_{00}(0).
  \end{split}
\end{equation}

Next, let us calculate the initial condition of ODE. We have
\begin{equation}
\begin{split}
  a_{mn}(0)&=\frac{1}{\sigma_m\sigma_n}  \int_{0}^1\int_{0}^1\phi(x,y,0)\cos(m\pi x)\cos(n\pi y)dxdy\\
&=\frac{1}{\sigma_m\sigma_n}  \int_{0}^{\tfrac{1}{2}}\cos(m\pi x)dx\int_0^1 \cos(n\pi y)dy\\
\end{split}
\end{equation}
Thus,
\begin{equation}
 a_{mn}(0)= \begin{cases}
\frac{1}{\sigma_m}\frac{\sin(\frac{m\pi}{2})}{m\pi}=\frac{2}{m\pi}& \text{ if } m=4k+1, n=0\\
\frac{1}{\sigma_m}\frac{\sin(\frac{m\pi}{2})}{m\pi}=-\frac{2}{m\pi}& \text{ if } m=4k+3, n=0\\
\frac{1}{2}& \text{ if } m=0, n=0\\
0 & \text{ otherwise }      
  \end{cases}
\end{equation}

\textbf{First Part:}
Now, let us derive the ODE. We have
\begin{equation*}
  \begin{split}
    \sum_{i=1}^\infty\sum_{j=1}^\infty \frac{da_{ij}(t)}{dt}\cos(i\pi x)\cos(j\pi y)&+\sin(\pi x)\cos(\pi y)\sum_{i=1}^\infty\sum_{j=1}^\infty a_{ij}(t)(-i\pi)\sin(i\pi x)\cos(j\pi y)\\
&-\cos(\pi x)\sin(\pi y)\sum_{i=1}^\infty\sum_{j=1}^\infty a_{ij}(t)(-j\pi)\cos(i\pi x)\sin(j\pi y)\\
=&-\kappa\pi^2\sum_{i=1}^\infty\sum_{j=1}^\infty a_{ij}(t)(i^2+j^2)\cos(i\pi x)\cos(j\pi y)
  \end{split}
\end{equation*}
i.e.,
\begin{equation*}
  \begin{split}
    \sum_{i=1}^\infty\sum_{j=1}^\infty \frac{da_{ij}(t)}{dt}\cos(i\pi x)\cos(j\pi y)&+\sum_{i=1}^\infty\sum_{j=1}^\infty a_{ij}(t)(-i\pi)\sin(\pi x)\sin(i\pi x)\cos(\pi y)\cos(j\pi y)\\
&+\sum_{i=1}^\infty\sum_{j=1}^\infty a_{ij}(t)(j\pi)\cos(\pi x)\cos(i\pi x)\sin(\pi y)\sin(j\pi y)\\
=&-\kappa\pi^2\sum_{i=1}^\infty\sum_{j=1}^\infty a_{ij}(t)(i^2+j^2)\cos(i\pi x)\cos(j\pi y)
  \end{split}
\end{equation*}
Using test function $\cos(m\pi x)\cos(n\pi y)$, we have 
\begin{equation}
  \frac{da_{mn}(t)}{dt}+\sum_{i=1}^\infty\sum_{j=1}^\infty a_{ij}(t)\frac{\pi}{\sigma_m\sigma_n}[-iA_iB_j+jC_iD_j]=-\kappa\pi^2(m^2+n^2)a_{mn}(t)
\end{equation}
where 
\begin{equation}
  \begin{split}
    A_i&=\int_0^1\cos(m\pi x)\sin(\pi x)\sin(i\pi x)dx~~B_j=\int_0^1\cos(n\pi y)\cos(\pi y)\cos(j\pi y)dy\\
    C_i&=\int_0^1\cos(m\pi x)\cos(\pi x)\cos(i\pi x)dx~~D_j=\int_0^1\cos(n\pi y)\sin(\pi y)\sin(j\pi y)dy
  \end{split}
\end{equation}

One can find that $A=D, B=C$ if $m=n, i=j.$

If $m=0$, then 
\begin{equation}
  A_i=
  \begin{cases}
    \frac{1}{2}& \text{ if } i=1\\
0 &\text{ otherwise }
  \end{cases}
~~~~
  C_i=
  \begin{cases}
    \frac{1}{2}& \text{ if } i=1\\
0 &\text{ otherwise }
  \end{cases}
\end{equation}

If $n=0$, then
\begin{equation}
  B_j=
  \begin{cases}
    \frac{1}{2}& \text{ if } j=1\\
0 &\text{ otherwise }
  \end{cases}
~~~~
  D_j=
  \begin{cases}
    \frac{1}{2}& \text{ if } j=1\\
0 &\text{ otherwise }
  \end{cases}
\end{equation}

If $m\neq 0, n\neq 0$, then 
\begin{equation}
  A_i=
  \begin{cases}
    \frac{1}{4}& \text{ if } i=m+1\\
-\frac{1}{4}&\text{ if } i= m-1 \text{ and }i>0\\
\frac{1}{4}&\text{ if } i = 1-m \text{ and }i>0\\
0 &\text{ otherwise }
  \end{cases}
~~~~
  B_j=
  \begin{cases}
    \frac{1}{4}& \text{ if } j=n\pm1, 1-n\text{ and }j>0\\
\frac{1}{2} & \text{ if } j =0 \text{ and } n=1\\
0 &\text{ otherwise }
  \end{cases}
\end{equation}

\begin{equation}
  C_i=
  \begin{cases}
    \frac{1}{4}& \text{ if } i=m\pm1, 1-m \text{ and } i>0\\
\frac{1}{2}&\text{ if } i =0 \text{ and } m=1\\
0 &\text{ otherwise }
  \end{cases}
~~~~
  D_j=
  \begin{cases}
    \frac{1}{4}& \text{ if } j=n+1\\
-\frac{1}{4}&\text{ if } j = n-1 \text{ and } j>0\\
\frac{1}{4}&\text{ if } j = 1-n \text{ and } j>0\\
0 &\text{ otherwise }
  \end{cases}
\end{equation}

\textbf{Second Part:} Now let us derive the ODE for the switching part, and let us derive the ODE. We have
\begin{equation*}
  \begin{split}
    \sum_{i=1}^\infty\sum_{j=1}^\infty \frac{da_{ij}(t)}{dt}\cos(i\pi x)\cos(j\pi y)&-\sin(2\pi x)\cos(\pi y)\sum_{i=1}^\infty\sum_{j=1}^\infty a_{ij}(t)(-i\pi)\sin(i\pi x)\cos(j\pi y)\\
&+2\cos(2\pi x)\sin(\pi y)\sum_{i=1}^\infty\sum_{j=1}^\infty a_{ij}(t)(-j\pi)\cos(i\pi x)\sin(j\pi y)\\
=&-\kappa\pi^2\sum_{i=1}^\infty\sum_{j=1}^\infty a_{ij}(t)(i^2+j^2)\cos(i\pi x)\cos(j\pi y)
  \end{split}
\end{equation*}
i.e.,
\begin{equation*}
  \begin{split}
    \sum_{i=1}^\infty\sum_{j=1}^\infty \frac{da_{ij}(t)}{dt}\cos(i\pi x)\cos(j\pi y)&+\sum_{i=1}^\infty\sum_{j=1}^\infty a_{ij}(t)(i\pi)\sin(2\pi x)\sin(i\pi x)\cos(\pi y)\cos(j\pi y)\\
&+\sum_{i=1}^\infty\sum_{j=1}^\infty a_{ij}(t)(-2j\pi)\cos(2\pi x)\cos(i\pi x)\sin(\pi y)\sin(j\pi y)\\
=&-\kappa\pi^2\sum_{i=1}^\infty\sum_{j=1}^\infty a_{ij}(t)(i^2+j^2)\cos(i\pi x)\cos(j\pi y)
  \end{split}
\end{equation*}
Using test function $\cos(m\pi x)\cos(n\pi y)$, we have 
\begin{equation}
  \frac{da_{mn}(t)}{dt}+\sum_{i=1}^\infty\sum_{j=1}^\infty a_{ij}(t)\frac{\pi}{\sigma_m\sigma_n}[i\tilde{A}_i\tilde{B}_j-2j\tilde{C}_i\tilde{D}_j]=-\kappa\pi^2(m^2+n^2)a_{mn}(t)
\end{equation}
where 
\begin{equation*}
  \begin{split}
    \tilde{A}_i&=\int_0^1\cos(m\pi x)\sin(2\pi x)\sin(i\pi x)dx\\
    \tilde{B}_j&=\int_0^1\cos(n\pi y)\cos(\pi y)\cos(j\pi y)dy\\
    \tilde{C}_i&=\int_0^1\cos(m\pi x)\cos(2\pi x)\cos(i\pi x)dx\\
    \tilde{D}_j&=\int_0^1\cos(n\pi y)\sin(\pi y)\sin(j\pi y)dy
  \end{split}
\end{equation*}

If $m=0$, then 
\begin{equation*}
  \tilde{A}_i=
  \begin{cases}
    \frac{1}{2}& \text{ if } i=2\\
0 &\text{ otherwise }
  \end{cases}
~~~~
  \tilde{C}_i=
  \begin{cases}
    \frac{1}{2}& \text{ if } i=2\\
0 &\text{ otherwise }
  \end{cases}
\end{equation*}

If $n=0$, then
\begin{equation*}
  \tilde{B}_j=
  \begin{cases}
    \frac{1}{2}& \text{ if } j=1\\
0 &\text{ otherwise }
  \end{cases}
~~~~
  \tilde{D}_j=
  \begin{cases}
    \frac{1}{2}& \text{ if } j=1\\
0 &\text{ otherwise }
  \end{cases}
\end{equation*}

If $m\neq 0, n\neq 0$, then 
\begin{equation*}
  \tilde{A}_i=
  \begin{cases}
    \frac{1}{4}& \text{ if } i=m+2\\
-\frac{1}{4}&\text{ if } i=m-2 \text{ and } i>0\\
\frac{1}{4}&\text{ if } i = 2-m \text{ and } i>0\\
0 &\text{ otherwise }
  \end{cases}
~~~~
  \tilde{B}_j=
  \begin{cases}
    \frac{1}{4}& \text{ if } j=n\pm1, 1-n \text{ and } j>0\\
\frac{1}{2} & \text{ if } j=0 \text{ and } n=1\\
0 &\text{ otherwise }
  \end{cases}
\end{equation*}

\begin{equation*}
  \tilde{C}_i=
  \begin{cases}
    \frac{1}{4}& \text{ if } i=m\pm2, 2-m \text{ and } i>0\\
\frac{1}{2} & \text{ if } i=0 \text{ and } m=2\\
0 &\text{ otherwise }
  \end{cases}
~~~~
  \tilde{D}_j=
  \begin{cases}
    \frac{1}{4}& \text{ if } j=n+1\\
-\frac{1}{4}&\text{ if } j = n-1 \text{ and }j>0\\
\frac{1}{4}&\text{ if } j = 1-n\text{ and } j>0\\
0 &\text{ otherwise }
  \end{cases}
\end{equation*}

\section{Conclusion}
The advection diffusion equation is a combination of the advection and diffusion equations, and describes physical phenomena where particles, energy, or other physical quantities are transferred inside a physical system due to two processes: advection and diffusion. In this paper, motivated by the problem of microfluidic heat transfer, we propose a PDE-based optimization to identify the optimal incompressible velocity fields in 2D domain. We applied the discretization-then-optimization method. In particular, we use spectral method to discretize the model to obtain an ODE based optimization to solve this optimization model.  
Spectral methods are a class of techniques used in applied mathematics  to numerically solve certain partial differential equations. The idea is to expand the solution of the differential equation as a sum of certain basis functions in infinite terms and then to choose the coefficients in the sum in order to approximate the differential equation as close as possible.
The novelty of this paper is that this way significantly reduces the complexity of the discretization optimization model, and gives a more accurate approximation of the original PDE based optimization. Some theoretical results are obtained about the boundness of the model coefficients. Future research will be focused on the invariant set of the incompressible velocity.

\bibliographystyle{plain}
\bibliography{myref}

\end{document}